\documentclass{rmmcart}
\title{Multiplicative Lie-type derivations on Alternative Rings}

\author{Bruno Leonardo Macedo Ferreira}
\address{{\it Federal Technological University of Paran\'{a}\\
Professora Laura Pacheco Bastos Avenue, 800\\
85053-510, Guarapuava, Brazil.}}
\email{brunoferreira@utfpr.edu.br}

\author{Henrique Guzzo Jr.}
\address{{\it Institute of Mathematics and Statistics, 
University of S\~{a}o Paulo\\
Mat\~{a}o Street, 1010\\
05508-090, S\~{a}o Paulo, Brazil.}}
\email{guzzo@ime.usp.br}

\author{Feng Wei}
\address{{\it School of Mathematics and Statistics, Beijing Institute of Technology,
Beijing\\
100081, P. R. China.}}
\email{daoshuo@bit.edu.cn}

\date{04, 10, 2019}


\usepackage{amssymb}
\usepackage[latin1]{inputenc}

\theoremstyle{plain}
\newtheorem{theorem}{Theorem}[section]
\newtheorem{lemma}[theorem]{Lemma}
\newtheorem{proposition}[theorem]{Proposition}

\newtheorem{corollary}[theorem]{Corollary}

\theoremstyle{definition}

\newtheorem{remark}[theorem]{Remark}

\allowdisplaybreaks[4]

\def\R{{\mathfrak R}\, }
\def\D{{\mathfrak D}\, }

\def\R{{\mathfrak R}\, }

\begin{document}


\begin{abstract}
Let $\R$ be an alternative ring containing a nontrivial idempotent and $\D$ be a multiplicative Lie-type derivation 
from $\R$ into itself. Under certain assumptions on $\R$, we prove that $\D$ is almost additive. 
Let $p_n(x_1, x_2, \cdots, x_n)$ be the $(n-1)$-th commutator defined by $n$ indeterminates $x_1, \cdots, x_n$. 
If $\R$ is a unital alternative ring with a nontrivial idempotent and is $\{2,3,n-1,n-3\}$-torsion free, it is shown under certain condition 
of $\R$ and $\D$, that $\D=\delta+\tau$, where $\delta$ is a derivation and 
$\tau\colon\R\longrightarrow{\mathcal Z}(\R)$ such that $\tau(p_n(a_1,\ldots,a_n))=0$ for all $a_1,\ldots,a_n\in\R$.
\end{abstract}
\maketitle

\vspace{0.5 in}

\section{Introduction and Preliminaries}

Let $\mathfrak{A}$ be an associative ring. We define the {\it Lie product} $[x,y]:=xy-yx$ and {\it Jordan product} $x\circ y:=xy+yx$ for all $x,y\in \mathfrak{A}$. 
Then $(\mathfrak{A},[\ ,\ ])$  becomes a Lie algebra and $(\mathfrak{A},\circ)$ is a Jordan algebra. It is a fascinating topic to study the 
connection between the associative, Lie and Jordan structures on $\mathfrak{A}$. In this field, two classes of 
mappings are of crucial importance. One of them consists of mappings, preserving a type of product, 
for example, Jordan homomorphisms and Lie homomorphisms. The other one is formed by differential operators, 
satisfying a type of Leibniz formulas, such as Jordan derivations and Lie derivations.
In the AMS Hour Talk of 1961, Herstein proposed many problems concerning the structure of 
Jordan and Lie mappings in associative simple and prime rings~\cite{Her}. Roughly speaking, 
he conjectured that these mappings are all of the proper or standard forms. The renowned 
Herstein's Lie-type mapping research program was formulated since then. Martindale 
gave a major force in this program under the assumption that the rings contain some 
nontrivial idempotents \cite{Mar64}. The first idempotent-free result on Lie-type mappings was obtained 
by Bre\v sar in \cite{Bresar}. The structures of derivations and Lie derivations on (non-)associative rings 
were studied systematically by many people 
(cf. \cite{Abdu, Benkovic, BenkovicEremita, Bresar, chang, changd, FerGur1, FerGur2, Fos, KaygorodovPopov1, KaygorodovPopov2, Her, Mar64}).
It is obvious that every derivation is a Lie derivation. But the converse is in general not true. 
A basic question towards Lie derivations of the associative algebras is that whether they 
can be decomposed into the sum of a derivation and a central-valued mapping, see \cite{Abdu, 
Benkovic, BenkovicEremita, Bresar, chang, changd, Fos, Mar64} and references therein.
In this paper, we will address the structure of Lie derivations without additivity on alternative rings.

Let $\mathfrak{R}$ and $\mathfrak{R}'$ be two rings (not necessarily associative) 
and $\varphi\colon\mathfrak{R}\longrightarrow\mathfrak{R}'$ be a mapping, 
we call $\varphi$ is \textit{additive} if $\varphi(a+b)=\varphi(a)+\varphi(b)$, \textit{almost additive} 
if $\varphi(a+b)-\varphi(a)-\varphi(b)\in\mathcal{Z}(\mathfrak{R})$, 
\textit{multiplicative} if $\varphi(ab)=\varphi(a)\varphi(b)$,  for all $a,b\in\mathfrak{R}$. 
Let $\mathfrak{R}$ be a ring with commutative centre $\mathcal{Z}(\mathfrak{R})$ and $\left[x_1,x_2\right] = x_1x_2-x_2x_1$ 
denote the usual Lie product of $x_1$ and $x_2$.
Let us define the following sequence of polynomials: 
$$p_1(x) = x\, \,  \text{and}\, \,  p_n(x_1, x_2, \ldots , x_n) = [p_{n-1}(x_1, x_2, \ldots , x_{n-1}) , x_n]$$ for all integers $n \geq 2$. 
Thus, $p_2(x_1, x_2) = [x_1, x_2], \  p_3 (x_1, x_2, x_3) = [[x_1, x_2] , x_3]$, etc. Let $n \geq 2$ be an integer. 
A mapping (not necessarily additive) $\D : \R \longrightarrow \R$ is called a \textit{multiplicative Lie n-derivation} if
\begin{eqnarray}\label{ident1}
\D(p_n (x_1, x_2, . . . ,x_n)) = \sum_{i = 1}^{n} p_n (x_1, x_2, . . . , x_{i-1}, \D(x_i), x_{i+1}, . . ., x_n). 
\end{eqnarray}
Lie $n$-derivations were introduced by Abdullaev \cite{Abdu}, where the form of Lie $n$-derivations of a 
certain von Neumann algebra was described. According to the definition, each multiplicative Lie derivation 
is a multiplicative Lie $2$-derivation and each multiplicative Lie triple derivation is a multiplicative Lie $3$-derivation. 
Fo\v sner et al \cite{Fos} showed that every multiplicative Lie $n$-derivation from an associative algebra $\mathcal{A}$ into itself is a 
multiplicative Lie $(n+k(n-1))$-derivation for each $k \in \mathbb{N}_0$. Multiplicative Lie $2$-derivations, Lie $3$-derivations 
and Lie $n$-derivations are collectively referred to as \textit{multiplicative Lie-type derivations}. 

A ring $\mathfrak{R}$ is said to be {\it alternative} if $(x,x,y)=0=(y,x,x)$ for all $x,y\in \mathfrak{R}$, and 
\textit{flexible} is $(x,y,x)=0$ for all $x,y\in\mathfrak{R}$, where $(x,y,z)=(xy)z-x(yz)$ is the associator of $x,y,z\in\mathfrak{R}$. 
It is known that alternative rings are flexible. An alternative ring $\mathfrak{R}$ is called {\it k-torsion free} if $k\,x=0$ 
implies $x=0,$ for any $x\in \mathfrak{R},$ where $k\in{\mathbb Z},\, k>0$, and {\it prime} if $\mathfrak{AB} \neq 0$ 
for any two nonzero ideals $\mathfrak{A},\mathfrak{B}\subseteq \mathfrak{R}$.
The {\it nucleus} $\mathcal{N}(\mathfrak{R})$ and the {\it commutative center} $\mathcal{Z}(\mathfrak{R})$ are defined by:
\begin{center}
$\mathcal{N}(\mathfrak{R})=\{r\in \mathfrak{R}\mid (x,y,r)=0=(x,r,y)=(r,x,y) \hbox{ for all }x,y\in \mathfrak{R}\}$ and 
$\mathcal{Z}(\mathfrak{R})=\{r\in \mathfrak{R} \mid [r, x] = 0 \hbox{ for all }x \in \mathfrak{R}\}$ 
\end{center}

By \cite[Theorem 1.1]{bruth} we have,

\begin{theorem}\label{meu}
Let $\mathfrak{R}$ be a $3$-torsion free alternative ring. So
$\mathfrak{R}$ is a prime ring if and only if $a\mathfrak{R} \cdot b=0$ (or $a \cdot \mathfrak{R}b=0$) 
implies that $a = 0$ or $b =0$ for $a, b \in \mathfrak{R}$. 
\end{theorem}

A nonzero element $e_{1}\in \mathfrak{R}$ is called an {\it idempotent} if $e_1^2=e_1$ and the idempotent $e_1$ is 
a {\it nontrivial idempotent} if $e_1$ is not the multiplicative identity element of $\mathfrak{R}$. Let us consider 
$\mathfrak{R}$ an alternative ring and fix a nontrivial idempotent $e_{1}\in\mathfrak{R}$. Let \mbox{$e_2 \colon\mathfrak{R}\rightarrow\mathfrak{R}$} 
and $e'_2 \colon\mathfrak{R}\rightarrow\mathfrak{R}$ be linear operators given by $e_2(a)=a-e_1a$ and $e_2'(a)=a-ae_1$. 
Clearly, $e_2^2=e_2\circ e_2=e_2,$ $(e_2')^2=e_2'$. Note that if $\mathfrak{R}$ has a unity, then 
$e_2=1-e_1\in \mathfrak{R}$. Let us denote $e_2(a)$ by $e_2a$ and $e_2'(a)$ by $ae_2$. It is easy to 
see that $e_ia\cdot e_j=e_i\cdot ae_j~(i,j=1,2)$ for all $a\in \mathfrak{R}$. By \cite{He} we know that 
$\mathfrak{R}$ has a Peirce decomposition
$$\mathfrak{R}=\mathfrak{R}_{11}\oplus \mathfrak{R}_{12}\oplus
\mathfrak{R}_{21}\oplus \mathfrak{R}_{22},$$ where
$\mathfrak{R}_{ij}=e_{i}\mathfrak{R}e_{j}$ $(i,j=1,2)$, satisfying the following multiplicative relations:
\begin{enumerate}\label{asquatro}
\item [\it {\rm (i)}] $\mathfrak{R}_{ij}\mathfrak{R}_{jl}\subseteq\mathfrak{R}_{il}\
(i,j,l=1,2);$
\item [\it {\rm(ii)}] $\mathfrak{R}_{ij}\mathfrak{R}_{ij}\subseteq \mathfrak{R}_{ji}\
(i,j=1,2);$
\item [\it {\rm(iii)}] $\mathfrak{R}_{ij}\mathfrak{R}_{kl}=0,$ if $j\neq k$ and
$(i,j)\neq (k,l),\ (i,j,k,l=1,2);$
\item [\it {\rm(iv)}] $x_{ij}^{2}=0,$ for all $x_{ij}\in \mathfrak{R}_{ij}\ (i,j=1,2;~i\neq j).$
\end{enumerate}

The first result about the additivity of mappings on rings was given by Martindale III in \cite{Mar69}, 
he established a condition on a ring $\mathfrak{R}$ such that every multiplicative isomorphism on $\mathfrak{R}$ 
is additive. In \cite{chang, changd}, Li and his coauthors also considered the 
almost additivity of maps for the case of Lie multiplicative mappings and Lie $3$-derivation on 
associative rings. They proved 

\begin{theorem}\label{cha}
Let $\mathfrak{R}$ be an associative ring containing a nontrivial idempotent $e_1$ and satisfying the
following condition: $(\mathbb{Q})$ If $A_{11}B_{12} = B_{12}A_{22}$ for all $B_{12} \in \mathfrak{R}_{12}$, then $A_{11} + A_{22} \in \mathcal{Z}(\mathfrak{R})$. Let $\mathfrak{R}'$ be another ring. Suppose that a bijection map $\Phi\colon \mathfrak{R} \rightarrow \mathfrak{R}'$ satisfies
$$\Phi([A,B]) = [\Phi(A),\Phi(B)]$$
for all $A,B \in \mathfrak{R}$. Then $\Phi(A + B) = \Phi(A) + \Phi(B) + Z_{A,B}'$ for all $A,B \in \mathfrak{R}$, where $Z_{A,B}'$ is an element in the commutative centre $\mathcal{Z}(\mathfrak{R}')$ of $\mathfrak{R}'$ depending on $A$ and $B$.
\end{theorem}

and

\begin{theorem}\label{chad}
Let $\mathfrak{R}$ be an associative ring containing a nontrivial idempotent $e_1$ and satisfying the
following condition: $(\mathbb{Q})$ If $A_{11}B_{12} = B_{12}A_{22}$ for all $B_{12} \in \mathfrak{R}_{12}$, then $A_{11} + A_{22} \in \mathcal{Z}(\mathfrak{R})$. 
Suppose that a mapping $\delta\colon \mathfrak{R} \longrightarrow \mathfrak{R}$ satisfies
$$\delta([[A,B],C]) = [[\delta(A),B],C] + [[A,\delta(B)],C] + [[A,B],\delta(C)]$$
for all $A,B, C \in \mathfrak{R}$. Then there exists a $Z_{A,B}$ {\rm (}depending on $A$ and $B${\rm )} in $\mathcal{Z}(\mathfrak{R})$ such that $\delta (A+B) = \delta(A) + \delta(B) + Z_{A,B}$.
\end{theorem}

In \cite{FerGur1}, Ferreira and Guzzo investigated the additivity of Lie triple derivations. They obtained the
following result. 

\begin{theorem}\label{Fegu1}
Let $\R$ be an alternative ring. Suppose that $\R$ is a ring
containing a nontrivial idempotent $e_1$ which satisfies
\begin{enumerate}
\item[\it {\rm (i)}] If $[a_{11}+ a_{22}, \mathfrak{R}_{12}] = 0$, then $a_{11} + a_{22} \in \mathcal{Z}(\mathfrak{R}),$
\item[\it {\rm (ii)}] If $[a_{11}+ a_{22}, \mathfrak{R}_{21}] = 0$, then $a_{11} + a_{22} \in \mathcal{Z}(\mathfrak{R}).$
\end{enumerate} 
Then each multiplicative Lie triple derivation $\D$ of $\R$ into itself is
almost additive.
\end{theorem}

In a recent paper, Ferreira and Guzzo study the characterization of Lie $2$-derivation 
on alternative rings, see \cite{FerGur2}. They showed that

\begin{theorem} \label{Fegu2}
Let $\R$ be a unital $2$,$3$-torsion free alternative ring with nontrivial idempotents $e_1$, $e_2$ and with associated 
Peirce decomposition $\R = \R_{11} \oplus \R_{12} \oplus \R_{21} \oplus \R_{22}$. Suppose that $\mathfrak{R}$ satisfies the following conditions:
\begin{enumerate}
\item [{\rm (1)}] If $x_{ij}\R_{ji} = 0$, then $x_{ij} = 0$ ($i \neq j$);
\item [{\rm (2)}] If $x_{11}\R_{12} = 0$ or $\R_{21}x_{11} = 0$, then $x_{11} = 0$;
\item [{\rm (3)}] If $\R_{12}x_{22} = 0$ or $x_{22}\R_{21} = 0$, then $x_{22} = 0$;
\item [{\rm (4)}] If $z \in \mathcal{Z}(\mathfrak{R})$ with $z \neq 0$, then $z\R = \R$.
\end{enumerate}
Let $\D \colon  \R \longrightarrow \R$ be a multiplicative Lie derivation of $\mathfrak{R}$. 
Then $\D$ is the form $\delta + \tau$, where $\delta$ is an additive derivation of $\R$ and $\tau$ is a mapping from 
$\R$ into the commutative centre $\mathcal{Z}(\mathfrak{R})$,  which maps commutators into the zero if and only if
\begin{enumerate}
\item[(a)] $e_2\D(\R_{11})e_2 \subseteq \mathcal{Z}(\R) e_2,$
\item[(b)] $e_1\D(\R_{22})e_1 \subseteq \mathcal{Z}(\R) e_1.$
\end{enumerate}
\end{theorem}

Inspired by the above-mentioned results, we are planning to extend Theorem \ref{Fegu1} to an arbitrary 
mulitplicative Lie-type derivations in Section $2$. In the Section $3$, we give the characterization of multiplicative Lie-type derivations on alternative rings 
and study the structure of multiplicative Lie-type derivations on alternative rings, which can be considered as 
a natural generalization of Theorem \ref{Fegu2}.
 
\section{Almost Additivity of Multiplicative Lie-type Derivations}

We shall prove as follows the first main result of this paper.

\begin{theorem}\label{mainthmd} 
Let $\mathfrak{R}$ be an alternative ring with nontrivial idempotent $e_1$, 
$\mathcal{Z}(\mathfrak{R})$ be the commutative center of $\mathfrak{R}$ and $\D$ be a 
multiplicative Lie-type derivation of $\R$. Suppose that 
$\mathfrak{R}$ satisfies the following conditions:
\begin{enumerate}
\item[(i)] If $[a_{11}+ a_{22}, \mathfrak{R}_{12}] = 0$, then $a_{11} + a_{22} \in \mathcal{Z}(\mathfrak{R})$, 
\item[(ii)] If $[a_{11}+ a_{22}, \mathfrak{R}_{21}] = 0$, then $a_{11} + a_{22} \in \mathcal{Z}(\mathfrak{R})$.
\end{enumerate} 
Then $\D$ is almost additive.
\end{theorem}

As our goal is to generalize the result obtained in \cite{FerGur1}, the following Lemmas are generalizations of Lemmas that appear in \cite{FerGur1}. 
The hypotheses of the following lemmas are the same as the Theorem \ref{mainthmd}.
\vspace{2mm}

It is easy to see that $\D(0)=0$.

\begin{lemma}\label{lema4d} For any $a_{11} \in \R_{11}$, $b_{ij} \in \R_{ij}$, with $i \neq j$ there exist $z_{a_{11}, b_{ij}} \in \mathcal{Z}(\R)$ such
that,
$\D(a_{11} + b_{ij}) = \D(a_{11}) + \D(b_{ij}) + z_{a_{11}, b_{ij}}.$
\end{lemma}
\begin{proof}
We only prove the case of $i=1$, $j=2$ because the demonstration of the other cases is rather similar by using the condition (i) of the Theorem \ref{mainthmd}.
Let us set $t = \D(a_{11} + b_{12}) - \D(a_{11}) - \D(b_{12})$.
Then we get $p_n (t, e_1, . . . ,e_1) = 0$, which is due to the fact
$$
\begin{aligned}
\D(p_n (a_{11} + b_{12}, e_1, . . . ,e_1)) &= \D((-1)^{n+1}b_{12}) \\
&= \D(p_n (a_{11}, e_1, . . . ,e_1)) + \D(p_n (b_{12}, e_1, . . . ,e_1)).
\end{aligned}
$$
In view of the definition of $\D$, we have $(-1)^{n+1}t_{12} + t_{21} = 0$. 
Now we will use the condition (ii) of the Theorem \ref{mainthmd}. For any $c_{21} \in \R_{21}$, we know that 
$$
\begin{aligned}
\D(p_n (a_{11} + b_{12}, c_{21},e_1, . . . ,e_1)) &= \D(-c_{21}a_{11}) \\
& = \D(p_n (a_{11}, c_{21},e_1, . . . ,e_1))+ \D(p_n (b_{12}, c_{21},e_1, . . . ,e_1)).
\end{aligned}
$$
Now using the definition of $\D$ and $D(0)=0$, we obtain $[t_{11} + t_{22}, c_{21}] = p_n (t, c_{21},e_1, . . . ,e_1) = 0$. Therefore by condition (ii) of the Theorem \ref{mainthmd} we have $t_{11} + t_{22} \in \mathcal{Z}(\R)$. And hence $\D(a_{11} + b_{12}) = \D(a_{11}) + \D(b_{12}) + z_{a_{11},b_{12}}$.

\end{proof}

\begin{lemma}\label{lema5d}
For any $a_{12} \in \R_{12}$ and $b_{21} \in \R_{21}$, we have $\D(a_{12} + b_{21}) = \D(a_{12}) + \D(b_{21})$.
\end{lemma}

\begin{proof}
Firstly, observe that $(-1)^{n+1} a_{12} + b_{21} = p_n (e_1 + a_{12}, e_1 - b_{21},e_1, . . . ,e_1)$ for all $a_{12} \in \R_{12}$ and $b_{21} \in \R_{21}$.
By invoking Lemma \ref{lema4d}, we arrive at
$$
\begin{aligned}
\D((-1)^{n+1} a_{12} + b_{21}) &= \D(p_n (e_1 + a_{12}, e_1 - b_{21},e_1, . . . ,e_1)) \\
&=  p_{n}(\D(e_1 + a_{12}), e_1 - b_{21} , e_1, ... , e_1) +p_{n}(e_1 + a_{12}, \D(e_1 - b_{21}) , e_1, ... , e_1) \\
&\ \ \ + \sum_{i=3}^{n} p_{n}(e_1 + a_{12}, e_1 - b_{21} , e_1, ...,\D(e_1), ... , e_1) \\
&= \D(p_{n}(e_1, e_1 , e_1, ... , e_1)) + \D(p_{n}(e_1, -b_{21} , e_1, ... , e_1)) \\
&\ \ \ + \D(p_{n}(a_{12}, e_1 , e_1, ... , e_1)) + \D(p_{n}(a_{12},-b_{21} , e_1, ... , e_1)) \\
&= \D((-1)^{n+1} a_{12}) + \D(b_{21}).  
\end{aligned}
$$
In the case of $n$ is odd, then $\D(-a_{12} + b_{21}) = \D(-a_{12}) + \D(b_{21})$. However, this clearly implies 
that $\D(a_{12} + b_{21}) = \D(a_{12}) + \D(b_{21})$.  
\end{proof}

\begin{lemma}\label{lema6d}
For any $a_{ij}, b_{ij} \in \R_{ij}$ with $i \neq j$, we have $\D(a_{ij} + b_{ij}) = \D(a_{ij}) + \D(b_{ij})$.
\end{lemma}

\begin{proof}
Here we shall only prove the case $i=2$, $j=1$ because the proofs of the other cases are similar.
Note that $x_{ij}^{2}=0,$ for all $x_{ij}\in \mathfrak{R}_{ij}\ (i,j=1,2;~i\neq j)$. Thus we have
$$
a_{21} + b_{21} + 2(-1)^{n+1}a_{21}b_{21} = p_{n}(e_1 + a_{21}, e_1 - b_{21},e_1, ..., e_1).
$$
Now making use of Lemma \ref{lema4d} and \ref{lema5d} we get
$$
\begin{aligned}
\D(a_{21} + b_{21}) + \D(2(-1)^{n+1}a_{21}b_{21}) & = \D(a_{21} + b_{21} +2(-1)^{n+1}a_{21}b_{21}) \\
&= \D(p_{n}(e_1 + a_{21}, e_1 - b_{21},e_1, ..., e_1)) \\
&= p_{n}(\D(e_1 + a_{21}), e_1 - b_{21},e_1, ..., e_1) + p_{n}(e_1 + a_{21}, \D(e_1 - b_{21}),e_1, ..., e_1) \\
& \ \ \ + \sum_{i=3}^{n}p_{n}(e_1 + a_{21}, e_1 - b_{21}, e_1, ..., \D(e_1), ..., e_1) \\
&= p_n(\D(e_1) + \D(a_{21}), e_1 -b_{21},e_1, ... , e_1) \\
&\ \ \ + p_n(e_1 + a_{21}, \D(e_1) + \D(-b_{21}),e_1 ..., e_1) \\
&\ \ \ + \sum_{i=3}^{n}p_{n}(e_1 + a_{21}, e_1 - b_{21},e_1, ..., \D(e_1), ... , e_1) \\
&= \D(p_n(e_1,e_1, ..., e_1))+ \D(p_n(e_1, -b_{21},e_1, ... e_1)) + \D(p_n(a_{21}, e_1, e_1, ... , e_1)) \\
&\ \ \ + \D(p_n(a_{21},-b_{21}, e_1, ... ,e_1)) \\
&= \D(a_{21}) + \D(b_{21}) + \D((-1)^{n+1}2a_{21}b_{21}).
\end{aligned}
$$
For the case $i=1$, $j=2$, we only need to use  
$$(-1)^{n+1}(a_{12} + b_{12}) + 2a_{12}b_{12} = p_n(e_1 + a_{12}, e_1-b_{12}, e_1,..., e_1)
$$
together with Lemma  \ref{lema4d} and \ref{lema5d}.
\end{proof}

\begin{lemma}\label{lema7d}
For any $a_{ii}, b_{ii} \in \R_{ii}$, $i=1,2$, there exists a $z_{a_{ii},b_{ii}} \in \mathcal{Z}(\R)$ such that
$$\D(a_{ii} + b_{ii}) = \D(a_{ii}) + \D(b_{ii}) + z_{a_{ii},b_{ii}}.$$
\end{lemma}
\begin{proof}
Let us set $t = \D(a_{ii} + b_{ii}) - \D(a_{ii}) - \D(b_{ii})$. On the one hand,
$$
\begin{aligned}
0 &=\D(0) \\
&= \D(p_n(a_{ii} + b_{ii} , e_1, ..., e_1)) \\
&= p_n(\D(a_{ii} + b_{ii}), e_1, ..., e_1) + \sum_{i=2}^{n} p_n(a_{ii}+b_{ii}, e_1, ..., \D(e_1), ..., e_1).
\end{aligned}
$$
On the other hand,
$$
\begin{aligned}
0 &= \D(0) + \D(0) \\
&= \D(p_n(a_{ii}, e_1, ..., e_1)) + \D(p_n(b_{ii}, e_1, ..., e_1)) \\
&= p_n(\D(a_{ii}) + \D(b_{ii}), e_1, ..., e_1) + \sum_{i=2}^{n} p_n(a_{ii}+b_{ii}, e_1, ..., \D(e_1), ..., e_1).
\end{aligned}
$$
This implies that $p_n(t, e_1, ..., e_1) = 0$. That is $t_{12} = t_{21} = 0$.
For any $c_{ij} \in \R_{ij}$, with $i \neq j$, by Lemma \ref{lema6d}, we obtain
$$
\begin{aligned}
\D((-1)^{n+1}(a_{ii} + b_{ii})c_{ij}) &= \D((-1)^{n+1}a_{ii}c_{ij}) + \D((-1)^{n+1}b_{ii}c_{ij}) \\
&= \D(p_n(c_{ij}, a_{ii}, e_1, ..., e_1)) + \D(p_n(c_{ij}, b_{ii}, e_1, ..., e_1)) \\
&= p_n(\D(c_{ij}), a_{ii} + b_{ii}, e_1, ..., e_1) + p_n(c_{ij}, \D(a_{ii}) + \D(b_{ii}), e_1, ..., e_1) \\
&\ \ \ + \sum_{i=3}^{n}p_n(c_{ij}, a_{ii} + b_{ii}, e_1, ..., \D(e_1), ..., e_1).
\end{aligned}
$$
Now we also have,
$$
\begin{aligned}
\D((-1)^{n+1}(a_{ii} + b_{ii})c_{ij}) &= \D(p_n(c_{ij}, a_{ii}+b_{ii}, e_1, ... , e_1)) \\
&= p_n(\D(c_{ij}), a_{ii}+b_{ii}, e_1, ... , e_1) + p_n(c_{ij}, \D(a_{ii}+b_{ii}), e_1, ... , e_1) \\
&\ \ \ + \sum_{i=3}^{n}p_n(c_{ij}, a_{ii}+b_{ii}, e_1, ..., \D(e_1), ... , e_1).
\end{aligned}	
$$
Hence $p_n(c_{ij}, t, e_1, ..., e_1) = 0$. This give $[t_{11} + t_{22}, c_{ij}] = 0$ for all $c_{ij} \in \R_{ij}$ with $i \neq j$. 
By the conditions of Theorem \ref{mainthmd}, we get $t_{11} + t_{22} \in \mathcal{Z}(\mathfrak{R})$. 
Therefore $\D(a_{ii} + b_{ii}) = \D(a_{ii}) + \D(b_{ii}) + z_{a_{ii},b_{ii}}$. 
\end{proof}

\begin{lemma}\label{lema8d}
For any $a_{11} \in \R_{11}$, $b_{12} \in \R_{12}$, $c_{21} \in \R_{21}$, $d_{22} \in \R_{22}$, there exists a $z_{a_{11},b_{12}, c_{21}, d_{22}} \in \mathcal{Z}(\R)$ such that
$$\D(a_{11} + b_{12} + c_{21} + d_{22}) = \D(a_{11}) + \D(b_{12}) + \D(c_{21}) + \D(d_{22}) + z_{a_{11},b_{12}, c_{21}, d_{22}}.$$
\end{lemma}
\begin{proof}
Let us write $t = \D(a_{11} + b_{12} + c_{21} + d_{22}) - \D(a_{11}) - \D(b_{12}) - \D(c_{21}) - \D(d_{22})$. 
By the definition of $\D$ and Lemma \ref{lema5d} we know that $p_n(t,e_1, ..., e_1) = 0$. 
Indeed, 
$$
\begin{aligned} 
p_n(t,e_1, ..., e_1)&= p_n(\D(a_{11} + b_{12} + c_{21} + d_{22}) - \D(a_{11}) - \D(b_{12}) - \D(c_{21}) - \D(d_{22}),e_1, ..., e_1)\\
&= p_n(\D(a_{11} + b_{12} + c_{21} + d_{22}),e_1, ..., e_1) - p_n(\D(a_{11}),e_1, ..., e_1) \\
&\ \ \ - p_n(\D(b_{12}),e_1, ..., e_1) - p_n(\D(c_{21}),e_1, ..., e_1) - p_n(\D(d_{22}),e_1, ..., e_1)\\
&= \D(p_n(a_{11} + b_{12} + c_{21} + d_{22},e_1, ..., e_1))\\
&\ \ \ - \sum_{i=2}^{n} p_n(a_{11} + b_{12} + c_{21} + d_{22}, e_1, ..., \D(e_1), ..., e_1)\\
&\ \ \ - \left\{\D(p_n(a_{11},e_1, ..., e_1)) - \sum_{i=2}^{n} p_n(a_{11},e_1,..., \D(e_1), ..., e_1)\right\}\\
&\ \ \ - \left\{\D(p_n(b_{12},e_1, ..., e_1)) - \sum_{i=2}^{n} p_n(b_{12},e_1,..., \D(e_1), ..., e_1)\right\}\\
&\ \ \ - \left\{\D(p_n(c_{21},e_1, ..., e_1)) - \sum_{i=2}^{n} p_n(c_{21},e_1,..., \D(e_1), ..., e_1)\right\}\\
&\ \ \ - \left\{\D(p_n(d_{22},e_1, ..., e_1)) - \sum_{i=2}^{n} p_n(d_{22},e_1,..., \D(e_1), ..., e_1)\right\}\\
&= \D((-1)^{n+1}b_{12} + c_{21}) - \D((-1)^{n+1}b_{12}) - \D(c_{21})\\
&= 0.
\end{aligned}
$$
As $p_n(t,e_1, ..., e_1) = 0$, we conclude that $(-1)^{n+1}t_{12} + t_{21} = 0$.
Now for all $x_{12} \in \R_{12}$, by Lemma \ref{lema5d} and Lemma \ref{lema6d} we get 
$$
\begin{aligned}
&p_n(\D(a_{11} + b_{12} + c_{21} + d_{22}), x_{12},e_1, ..., e_1) + p_n(a_{11} + b_{12} + c_{21} + d_{22}, \D(x_{12}),e_1, ..., e_1) \\
&+ \sum_{i=3}^{n}p_n(a_{11} + b_{12} + c_{21} + d_{22}, x_{12},e_1,..., \D(e_1), ..., e_1) \\
&= \D(p_n(a_{11} + b_{12} + c_{21} + d_{22}, x_{12},e_1, ..., e_1)) \\
&= \D((-1)^{n+1}x_{12}d_{22} +(-1)^{n} a_{11}x_{12} +(-1)^{n} b_{12}x_{12}) \\
&= \D((-1)^{n+1}x_{12}d_{22} +(-1)^{n} a_{11}x_{12}) + \D((-1)^{n}b_{12}x_{12}) \\
&= \D((-1)^{n+1}x_{12}d_{22}) + \D((-1)^{n} a_{11}x_{12}) + \D((-1)^{n}b_{12}x_{12}) \\
&= \D(p_n(a_{11}, x_{12},e_1, ..., e_1)) + \D(p_n(b_{12}, x_{12},e_1, ..., e_1)) \\
&+ \D(p_n(c_{21}, x_{12},e_1, ..., e_1)) + \D(p_n(d_{22}, x_{12},e_1, ..., e_1))\\
&= p_n(\D(a_{11}) + \D(b_{12}) + \D(c_{21}) + \D(d_{22}), x_{12},e_1, ... , e_1) \\
&+ p_n(a_{11} + b_{12} + c_{21} + d_{22}, \D(x_{12}),e_1, ..., e_1) \\
&+ \sum_{i=3}^{n}p_n(a_{11} + b_{12} + c_{21} + d_{22}, x_{12},e_1, ..., \D(e_1), ..., e_1).
\end{aligned}
$$
We therefore have $p_n(\D(a_{11} + b_{12} + c_{21} + d_{22}), x_{12},e_1, ..., e_1)=p_n(\D(a_{11}) + \D(b_{12}) + \D(c_{21}) + \D(d_{22}), x_{12},e_1, ..., e_1)$. 
That is, $[t_{11} + t_{22}, x_{12}]=p_n(t,x_{12},e_1, ..., e_1) = 0$. Applying the condition (i) of Theorem \ref{mainthmd} yields $t = t_{11} + t_{22} \in \mathcal{Z}(\R)$.
Thus $\D(a_{11} + b_{12} + c_{21} + d_{22}) = \D(a_{11}) + \D(b_{12}) + \D(c_{21}) + \D(d_{22}) + z_{a_{11},b_{12}, c_{21}, d_{22}}$, where $z_{a_{11},b_{12}, c_{21}, d_{22}} \in \mathcal{Z}(\R)$.  
\end{proof}

\vspace{2mm}

We are ready to prove our Theorem \ref{mainthmd}.

\vspace{0,5cm}

\noindent {\bf Proof of Theorem 2.1}. Let $a, b \in \R$ with $a= a_{11} + a_{12} + a_{21} + a_{22}$ and $b= b_{11} + b_{12} + b_{21} + b_{22}$. By previous Lemmas we obtain
$$
\begin{aligned}
\D(a + b) &= \D(a_{11} + a_{12} + a_{21} + a_{22} + b_{11} + b_{12} + b_{21} + b_{22})\\
&= \D((a_{11} + b_{11}) + (a_{12} + b_{12}) + (a_{21}+ b_{21}) + (a_{22} + b_{22}))\\
&=\D(a_{11} + b_{11}) + \D(a_{12} + b_{12}) + \D(a_{21}+ b_{21}) + \D(a_{22} + b_{22}) + z_{1}\\
&= \D(a_{11}) +\D( b_{11}) + z_{2} + \D(a_{12}) + \D(b_{12}) + \D(a_{21})+ \D(b_{21}) +\D(a_{22}) + \D(b_{22})\\
&\ \ \  + z_{3} + z_{1}\\
&= (\D(a_{11}) +\D(a_{12}) +\D(a_{21})+\D(a_{22}))+ (\D( b_{11}) + \D( b_{12}) + \D( b_{21}) + \D( b_{22})) \\
&\ \ \ + (z_1 + z_2 + z_3)\\
&= \D(a_{11} + a_{12} + a_{21} + a_{22}) - z_{4} + \D(b_{11} + b_{12} + b_{21} + b_{22}) - z_{5} +(z_1 + z_2 + z_3)\\
&= \D(a) + \D(b) + (z_1 + z_2 + z_3 - z_4 -z_5)\\
&= \D(a) + \D(b) + z_{a,b}.
\end{aligned}
$$
This finishes the proof of Theorem \ref{mainthmd}.

\begin{corollary}
Let $\mathfrak{R}$ be an alternative ring.
Suppose that $\mathfrak{R}$ is a ring containing a nontrivial idempotent $e_1$ which satisfies:
\begin{enumerate}
\item[\it (i)] If $[a_{11}+ a_{22}, \mathfrak{R}_{12}] = 0$, then $a_{11} + a_{22} \in \mathcal{Z}(\mathfrak{R})$,
\item[\it (ii)] If $[a_{11}+ a_{22}, \mathfrak{R}_{21}] = 0$, then $a_{11} + a_{22} \in \mathcal{Z}(\mathfrak{R})$.
\end{enumerate}
Then every Lie $3$-derivation $\D$ of $\mathfrak{R}$ into itself is almost additive.
\end{corollary}

\begin{corollary}
Let $\mathfrak{R}$ be a $3$-torsion free prime alternative ring. Suppose that $\mathfrak{R}$ is an 
alternative ring containing a nontrivial idempotent $e_1$. Then every Lie $3$-derivation $\D$ of 
$\mathfrak{R}$ into itself is almost additive.
\end{corollary}
\begin{proof}
In \cite{FerGur1} the authors showed that any prime alternative ring satisfies the conditions of the Theorem \ref{mainthmd}. 
Hence the result holds true for $n=3$.

\end{proof}

\section{Characterization of Lie-type derivations on alternative rings}

In this section, we will characterize multiplicative Lie-type derivations on alternative rings and provide 
an essential structure theorem for multiplicative Lie-type derivations.
Henceforth, let $\R$ be a $\left\{2, 3, (n-1), (n-3)\right\}$-torsion free alternative ring satisfying the following conditions:
\begin{enumerate}\label{identi}
\item [(1)] If $x_{ij}\R_{ji} = 0$, then $x_{ij} = 0$ ($i \neq j$);
\item [(2)] If $x_{11}\R_{12} = 0$ or $\R_{21}x_{11} = 0$, then $x_{11} = 0$;
\item [(3)] If $\R_{12}x_{22} = 0$ or $x_{22}\R_{21} = 0$, then $x_{22} = 0$;
\item [(4)] If $z \in \mathcal{Z}$ with $z \neq 0$, then $z\R = \R$.
\end{enumerate}

We refer the reader to \cite{FerGur2} about the proofs of the following propositions.

\begin{proposition}
Any prime alternative ring satisfies $(1)$, $(2)$, $(3)$.   
\end{proposition}

\begin{proposition}\label{prop2}
Let $\R$ be a $2,3$-torsion free alternative ring satisfying the conditions $(1)$, $(2)$ and $(3)$.
\begin{enumerate}
\item [$(\spadesuit)$] If $[a_{11}+ a_{22}, \mathfrak{R}_{12}] = 0$, then $a_{11} + a_{22} \in \mathcal{Z}(\mathfrak{R})$,
\item [$(\clubsuit)$] If $[a_{11}+ a_{22}, \mathfrak{R}_{21}] = 0$, then $a_{11} + a_{22} \in \mathcal{Z}(\mathfrak{R})$.
\end{enumerate}
\end{proposition}

\begin{proposition}\label{prop3}
If $\mathcal{Z}(\R_{ij}) = \left\{a \in \R_{ij} ~ | ~ [a, \R_{ij}] = 0 \right\}$, then $\mathcal{Z}(\R_{ij}) \subseteq \R_{ij} + \mathcal{Z}(\R)$ with $i \neq j$.
\end{proposition}

The main result in this section reads as follows.

\begin{theorem}\label{mainthm} 
Let $\R$ be a unital $\left\{2,3,n-1,n-3\right\}$-torsion free alternative ring with nontrivial idempotents $e_1$, $e_2$ and 
with associated Peirce decomposition $\R = \R_{11} \oplus \R_{12} \oplus \R_{21} \oplus \R_{22}$. 
Suppose that $\mathfrak{R}$ satisfies the following conditions:
\begin{enumerate}
\item [{\rm (1)}] If $x_{ij}\R_{ji} = 0$, then $x_{ij} = 0$ ($i \neq j$);
\item [{\rm (2)}] If $x_{11}\R_{12} = 0$ or $\R_{21}x_{11} = 0$, then $x_{11} = 0$;
\item [{\rm (3)}] If $\R_{12}x_{22} = 0$ or $x_{22}\R_{21} = 0$, then $x_{22} = 0$;
\item [{\rm (4)}] If $z \in \mathcal{Z}(\mathfrak{R})$ with $z \neq 0$, then $z\R = \R$.
\end{enumerate}
Let $\D \colon  \R \longrightarrow \R$ be a multiplicative Lie-type derivation of $\mathfrak{R}$. 
Then $\D$ is the form $\delta + \tau$, where $\delta$ is an additive derivation of $\R$ and $\tau$ is a mapping from 
$\R$ into the commutative centre $\mathcal{Z}(\mathfrak{R})$, such that $\tau(p_n(a_1, a_2, ..., a_n)) = 0$ for all $a_1, a_2, ... , a_n \in \R$ if and only if
\begin{enumerate}
\item[(a)] $e_2\D(\R_{11})e_2 \subseteq \mathcal{Z}(\R) e_2,$
\item[(b)] $e_1\D(\R_{22})e_1 \subseteq \mathcal{Z}(\R) e_1.$
\item[(c)] $\D(\R_{ij}) \subseteq \R_{ij}$, $1 \leq i \neq j \leq 2.$
\end{enumerate}
\end{theorem}

\vspace{.1in}
The following Lemmas has the same hypotheses of Theorem \ref{mainthm} and we need these Lemmas for the proof of the first part this Theorem. 

Firstly, assume that the multiplicative Lie-type derivation $\D \colon \R \longrightarrow \R$ satisfies the conditions $(a)$, $(b)$ and 
$(c)$. Let $e_{1}$ be a nontrivial idempotent of $\mathfrak{R}$. We started with the following lemma. 

\begin{lemma}\label{lema2} $\D(e_1) - f_{y,z}(e_1) \in \mathcal{Z}(\R)$, with $y= \D(e_1)_{12} + \D(e_1)_{21}$, $z = e_1$ where $f_{y,z} \colon= [L_y, L_z] + [L_y, R_z] + [R_y, R_z]$ and $L$, $R$ are left and right multiplication operators, respectively.
\end{lemma}
\begin{proof}
In the case of $n$ is even,  we have 
$$
\begin{aligned}
\D(a_{12}) &= \D(p_n(e_1,a_{12},e_1,..., e_1)) \\
& = p_n(e_1,\D(a_{12}),e_1,..., e_1) + \sum_{i=2}^{n}p_n(e_1,a_{12},e_1,...,\D(e_1),..., e_1)\\
&=-a_{12}\D(e_1)e_1+e_1\D(e_1)a_{12} - a_{12}\D(e_1) + e_1\D(a_{12}) - \D(a_{12})e_1 \\
&\ \ \ + \sum_{i=3}^{n}p_n(e_1,a_{12},e_1,...,\D(e_1),...,e_1).
\end{aligned}
$$
Multiplying the left and right sides in the above equation by $e_1$ and $e_2$, respectively, we obtain
$$
\begin{aligned}
e_1 \D(a_{12})e_2 &= e_1\D(e_1)a_{12} - a_{12}\D(e_1)e_2 +e_1\D(a_{12})e_2+\sum_{i=3}^{n}(-1)^{n-1}[\D(e_1)_{11}+\D(e_1)_{22}, a_{12}].
\end{aligned}
$$
This implies
$$
-(n-3)[\D(e_1)_{11} + \D(e_1)_{22} , a_{12}] + 2\D(e_1)_{12}a_{12} = 0
$$ 
for all $a_{12} \in \R_{12}$. In light of $(\spadesuit)$ of Proposition \ref{prop2}, we assert that $\D(e_1)_{11} + \D(e_1)_{22} \in \mathcal{Z}(\R)$.
Taking $y= \D(e_1)_{12} + \D(e_1)_{21}$ and $z= e_1$ we see that $\D(e_1) - f_{y,z}(e_1) = \D(e_1)_{11} + \D(e_1)_{22} \in \mathcal{Z}(\R)$.

In the case of $n$ is odd, we get 
$$
\begin{aligned}
\D(a_{12}) &= \D(p_n(a_{12},e_1,..., e_1))\\
& = p_n(\D(a_{12}),e_1,..., e_1) + \sum_{i=2}^{n}p_n(a_{12},e_1,...,\D(e_1),..., e_1)\\
&= \D(a_{12})e_1-2e_1\D(a_{12})e_1 + e_1\D(a_{12})+ \sum_{i=2}^{n}p_n(a_{12},e_1,...,\D(e_1),...,e_1).
\end{aligned}
$$
Multiplying the left and right sides in the above equation by $e_1$ and $e_2$, respectively, we arrive at
$$
\begin{aligned}
e_1 \D(a_{12})e_2 &= e_1\D(a_{12})e_2 + \sum_{i=2}^{n}e_1p_n(a_{12},e_1,...,\D(e_1),...,e_1)e_2 \\
&= e_1\D(a_{12})e_2 -(n-1)[\D(e_1)_{11} + \D(e_1)_{22}, a_{12}].
\end{aligned}
$$
This gives that 
$(n-1)[\D(e_1)_{11} + \D(e_1)_{22}, a_{12}] = 0$ for all $a_{12} \in \R_{12}$. By $(\spadesuit)$ of Proposition \ref{prop2} we conclude that $\D(e_1)_{11} + \D(e_1)_{22} \in \mathcal{Z}(\R)$.
Taking $y= \D(e_1)_{12} + \D(e_1)_{21}$ and $z= e_1$ again,  we see that $\D(e_1) - f_{y,z}(e_1) = \D(e_1)_{11} + \D(e_1)_{22} \in \mathcal{Z}(\R)$. 
\end{proof}

Let us continue our discussions. It is worth noting that $f_{y,z} \colon= [L_y, L_z] + [L_y, R_z] + [R_y, R_z]$ is a derivation. According 
to \cite[Page 77]{sch}, we without loss of generality may assume that 
$\D(e_1) \in \mathcal{Z}(\R)$.

\begin{remark}\label{obs1}
If $\D(e_1) \in \mathcal{Z}(\R)$, then $\D(e_2) \in \mathcal{Z}(\R)$. Indeed, since
$$
\begin{aligned}
0 &= \D(p_n(e_2, e_1, ..., e_1) \\
&= p_n(\D(e_2),e_1,...,e_1) + \sum_{i=2}^{n}p_n(e_2,e_1,..., \D(e_1),...,e_1) \\
&= p_n(\D(e_2),e_1,...,e_1) \\
&=\D(e_2)e_1 -e_1\D(e_2)e_1 + (-1)^ne_1\D(e_2)e_1 +  (-1)^{n+1}e_1\D(e_2),
\end{aligned}
$$
we know that $e_1\D(e_2)e_2 = e_2\D(e_2)e_1 = 0$.  When $n$ is even,  for any $a_{12} \in \R_{12}$, we have
$$
\begin{aligned}
\D(a_{21}) &= \D(p_n(e_2,a_{21}, e_2...,e_2)) \\
&= p_n(\D(e_2), a_{21}, e_2,...,e_2) + p_n(e_2, \D(a_{21}),e_2,...,e_2)+ \sum_{i=3}^{n}p_n(e_2,a_{21},e_2,...,\D(e_2),...,e_2) \\
&= -(-1)^{n-2}a_{21}\D(e_2)_{11}+(-1)^{n-2}\D(e_2)_{22}a_{21} + e_2\D(a_{21})-\D(a_{21})e_2\\
&\ \ \  -(n-2)[\D(e_2)_{11} + \D(e_2)_{22}, a_{21}] \\
&= -(n-1)[\D(e_2)_{11} + \D(e_2)_{22}, a_{21}] + e_2\D(a_{21}) -\D(a_{21})e_2.
\end{aligned}
$$
Multiplying by $e_2$ and $e_1$ from the left and right sides in
the above equation, respectively, we arrive at
$-(n-1)[\D(e_2)_{11} + \D(e_2)_{22}, a_{21}] = 0$ for all $a_{21} \in \R_{21}$. 
This gives
$$
[\D(e_2)_{11} + \D(e_2)_{22}, a_{21}] = 0
$$
for all $a_{21} \in \R_{21}$, since the characteristic of $\R$ is not $n-1$. 
By $(\clubsuit)$ of Proposition \ref{prop2} it follows that $\D(e_2) = \D(e_2)_{11} + \D(e_2)_{22} \in \mathcal{Z}(\R)$. 
Now if $n$ is odd, then we have
$$
\begin{aligned}
\D(a_{21}) &= \D(p_n(a_{21},e_2...,e_2)) \\
&= p_n(\D(a_{21}), e_2,...,e_2) + \sum_{i=2}^{n}p_n(a_{21},e_2,...,\D(e_2),...,e_2) \\
&= -2e_2\D(a_{21})e_2 + e_2\D(a_{21})+\D(a_{21})e_2  -(n-1)[\D(e_2)_{11} + \D(e_2)_{22}, a_{21}].
\end{aligned}
$$
Multiplying by $e_2$ and $e_1$ from the left and right sides in
the above equation, respectively, we obtain the same result as $n$ is even.
\end{remark}

\begin{lemma}\label{lema3} $\D(\R_{ii}) \subseteq \R_{ii} + \mathcal{Z}(\R) \ (i = 1,2)$
\end{lemma}
\begin{proof}
We only show the case of $i = 1$, because the other case can be treated similarly. 
For each $a_{11} \in \R_{11}$, with $\D(a_{11}) = b_{11} + b_{12} + b_{21} + b_{22}$ we get
$$
\begin{aligned}
0 &= \D(p_n(a_{11}, e_1,...e_1)) \\
&= p_n(\D(a_{11}),e_1,...,e_1) + \sum_{i=2}^{n}p_n(a_{11},e_1,...,\D(e_1),...,e_1) \\
&= p_n(\D(a_{11}),e_1,...,e_1).
\end{aligned}
$$
It follows from this that $b_{12} = b_{21} = 0$. By $(a)$ of Theorem \ref{mainthm} we know that
$$\D(a_{11}) = b_{11} + e_2\D(a_{11})e_2 = b_{11} + ze_2 = b_{11} -e_1z + z \in \R_{11} + \mathcal{Z}(\R).$$
\end{proof}

\begin{lemma}\label{lema4}
$\D$ is an almost additive mapping. That is, for any $a, b \in \R$, $\D(a+b) - \D(a) - \D(b) \in \mathcal{Z}(\R)$.
\end{lemma}
\begin{proof}
Since $\R$ is a alternative ring satisfying the conditions $(1)$, $(2)$ and $(3)$, $\R$ satisfies 
$(\spadesuit)$ and $(\clubsuit)$ by Proposition \ref{prop2}. Now using Theorem \ref{mainthmd} we get $\D$ is an 
almost additive mapping. 
\end{proof}

Now let us define the mappings $\delta$ and $\tau$. By the item $(c)$ of Theorem \ref{mainthm} and 
Lemma \ref{lema3} we have

\begin{enumerate}
\item[(A)] if $a_{ij} \in \R_{ij}$, $i \neq j$, then $\D(a_{ij}) = b_{ij} \in \R_{ij}$,
\item[(B)] if $a_{ii} \in \R_{ii}$, then $\D(a_{ii}) = b_{ii} + z, b_{ii} \in \R_{ii}$, where $z$ is a central element.
 \end{enumerate} 
It should be remarked that $b_{ii}$ and $z$ in $({\rm B})$ are uniquely determined,  Indeed, if 
$\D(a_{ii}) = b'_{ii}+ z'$, $b'_{ii} \in \R_{ii}$, $z' \in \mathcal{Z}(\R)$. Then $b_{ii} - b'_{ii} \in \mathcal{Z}(\R)$. 
Taking into account the conditions $(2)$ and $(3)$, we assert that $b_{ii} = b'_{ii}$ and $z = z'$.
Now let us define a mapping $\delta$ of $\R$ according to the rule $\delta(a_{ij}) = b_{ij}, a_{ij} \in \R_{ij}$. 
For each $a = a_{11} + a_{12} + a_{21} + a_{22} \in \R$, we define $\delta(a) = \sum \delta(a_{ij})$. And a mapping 
$\tau$ of $\R$ into $\mathcal{Z}(\R)$ is then defined by
$$
\begin{aligned}
\tau(a) &= \D(a) - \delta(a) \\
&= \D(a) - (\delta(a_{11}) + \delta(a_{12}) + \delta(a_{21}) + \delta(a_{22})) \\
&= \D(a) - (b_{11} + b_{12} + b_{21} + b_{22}) \\
&= \D(a) - (\D(a_{11}) - z_{a_{11}} + \D(a_{12}) + \D(a_{21}) + \D(a_{22}) - z_{a_{22}}) \\
&= \D(a) - (\D(a_{11}) + \D(a_{12}) + \D(a_{21}) + \D(a_{22}) - (z_{a_{11}} + z_{a_{22}})) \\
&= \D(a) - (\D(a_{11}) + \D(a_{12}) + \D(a_{21}) + \D(a_{22})). 
\end{aligned}
$$
We need to show that $\delta$ and $\tau$ are the desired mappings. 

\begin{lemma}\label{lema5}
$\delta$ is an additive mapping.
\end{lemma}
\begin{proof}
We only need to prove that $\delta$ is an additive mapping on $\R_{ii}$. Let us choose any $a_{ii}, b_{ii} \in \R_{ii}$,
$$
\delta(a_{ii} + b_{ii}) - \delta(a_{ii}) - \delta(b_{ii}) = \D(a_{ii} + b_{ii}) - \tau(a_{ii} + b_{ii}) - \D(a_{ii}) + \tau(a_{ii}) - \D(b_{ii}) + \tau(b_{ii}).
$$
Thus, $\delta(a_{ii} + b_{ii}) - \delta(a_{ii}) - \delta(b_{ii}) \in \mathcal{Z}(\R) \cap \R_{ii} = \left\{0\right\}$.
\end{proof}

Let us next show that $\delta(ab) = \delta(a)b + a\delta(b)$ for all $a, b \in \R$.

\begin{lemma}\label{lema6}
For any $a_{ii}, b_{ii} \in \R_{ii}$, $a_{ij}, b_{ij} \in \R_{ij}$, $b_{ji} \in \R_{ji}$ and $b_{jj} \in \R_{jj}$ with $i \neq j$, we have
\begin{enumerate}
\item[(I)] $\delta(a_{ii}b_{ij}) = \delta(a_{ii})b_{ij} + a_{ii}\delta(b_{ij})$,
\item[(II)] $\delta(a_{ij}b_{jj}) = \delta(a_{ij})b_{jj} + a_{ij}\delta(b_{jj})$,
\item[(III)] $\delta(a_{ii}b_{ii}) = \delta(a_{ii})b_{ii} + a_{ii}\delta(b_{ii})$,
\item[(IV)] $\delta(a_{ij}b_{ij}) = \delta(a_{ij})b_{ij} + a_{ij}\delta(b_{ij})$,
\item[(V)] $\delta(a_{ij}b_{ji}) = \delta(a_{ij})b_{ji} + a_{ij}\delta(b_{ji}).$
\end{enumerate}
\end{lemma}
\begin{proof}
Let us begin with $({\rm I})$
$$
\begin{aligned}
\delta(a_{ii}b_{ij}) &= \D(a_{ii}b_{ij})\\
&= \D(p_n(a_{ii}, b_{ij},e_j,...,e_j)) \\
&= p_n(\D(a_{ii}), b_{ij},e_j,...,e_j) + p_n(a_{ii}, \D(b_{ij}),e_j,...e_j) \\
&= p_n(\delta(a_{ii}), b_{ij},e_j,...,e_j) + p_n(a_{ii}, \delta(b_{ij}),e_j,...,e_j) \\
&= \delta(a_{ii})b_{ij} + a_{ii}\delta(b_{ij}).
\end{aligned} 
$$
Let us see $({\rm II})$
$$
\begin{aligned}
\delta(a_{ij}b_{jj}) &= \D(a_{ij}b_{jj})\\
&= \D(p_n(a_{ij}, b_{jj},e_j,...,e_j)) \\
&= p_n(\D(a_{ii}), b_{ij},e_j,...,e_j) + p_n(a_{ii}, \D(b_{ij}),e_j,...,e_j) \\
&= p_n(\delta(a_{ii}), b_{ij},e_j,...,e_j) + p_n(a_{ii}, \delta(b_{ij}),e_j,...,e_j) \\
&= \delta(a_{ij})b_{jj} + a_{ij}\delta(b_{jj}).
\end{aligned} 
$$
We next show $({\rm III})$. By linearization of flexible identity and $({\rm I})$ we get
$$
\delta((a_{ii}b_{ii})r_{ij}) = \delta(a_{ii}b_{ii})r_{ij} + (a_{ii}b_{ii})\delta(r_{ij}).
$$ 
On the other hand,
$$
\delta(a_{ii}(b_{ii}r_{ij})) = \delta(a_{ii})b_{ii}r_{ij} + a_{ii}\delta(b_{ii}r_{ij}) = \delta(a_{ii})b_{ii}r_{ij} + a_{ii}(\delta(b_{ii})r_{ij} + b_{ii}\delta(r_{ij})).
$$
Considering the facts $(a_{ii}b_{ii})r_{ij} = a_{ii}(b_{ii}r_{ij})$ and $(a_{ii}b_{ii})\delta(r_{ij}) = a_{ii}(b_{ii}\delta(r_{ij}))$, we obtain
$$
(\delta(a_{ii}b_{ii}) - \delta(a_{ii})b_{ii} - a_{ii}\delta(b_{ii}))r_{ij} = 0
$$
for all $r_{ij} \in \R_{ij}$. And hence $\delta(a_{ii}b_{ii}) = \delta(a_{ii})b_{ii} + a_{ii}\delta(b_{ii})$.

Let us prove $({\rm IV})$. 
$$
\begin{aligned}
2\delta(a_{ij}b_{ij}) &= \delta(2a_{ij}b_{ij}) = \D(2a_{ij}b_{ij}) \\
&= \D(p_n(a_{ij}, b_{ij},e_i,...,e_i)) = p_n(\D(a_{ij}), b_{ij},e_i,...,e_i) + p_n(a_{ij}, \D(b_{ij}),e_i,...,e_i)\\
&= p_n(\delta(a_{ij}), b_{ij},e_i,...,e_i)+ p_n(a_{ij}, \delta(b_{ij}),e_i,...,e_i) \\
&= \delta(a_{ij})b_{ij} - b_{ij}\delta(a_{ij}) + a_{ij}\delta(b_{ij}) - \delta(b_{ij})a_{ij} \\
&= 2(\delta(a_{ij})b_{ij} + a_{ij}\delta(b_{ij})) 
\end{aligned}
$$
Since $\R$ is $2$-torsion free, we see that $\delta(a_{ij}b_{ij}) = \delta(a_{ij})b_{ij} + a_{ij}\delta(b_{ij})$. 
And finally we show the $({\rm V})$. We get
$$
\begin{aligned}
\tau(p_n(a_{ij}, b_{ji},c_{ij},e_j,...,e_j)) &= \D(p_n(a_{ij}, b_{ji},c_{ij},e_j,...,e_j)) - \delta(p_n(a_{ij}, b_{ji},c_{ij},e_j,...,e_j)) \\
&= p_n(\D(a_{ij}), b_{ji},c_{ij},e_j,...,e_j) + p_n(a_{ij}, \D(b_{ji}),c_{ij},e_j,...,e_j)\\
&\ \ \ +p_n(a_{ij}, b_{ji},\D(c_{ij}),e_j,...,e_j) - \delta((a_{ij}b_{ji})c_{ij} - c_{ij}(b_{ji}a_{ij})) \\
&= p_n(\delta(a_{ij}), b_{ji},c_{ij},e_j,...,e_j) + p_n(a_{ij}, \delta(b_{ji}),c_{ij},e_j,...,e_j)\\
&\ \ \ +p_n(a_{ij}, b_{ji},\delta(c_{ij}),e_j,...,e_j)- \delta((a_{ij}b_{ji})c_{ij}) - \delta(c_{ij}(b_{ji}a_{ij}))\\
&= (\delta(a_{ij})b_{ji})c_{ij} + c_{ij}(b_{ji}\delta(a_{ij})) + (a_{ij}\delta(b_{ji}))c_{ij} + c_{ij}( \delta(b_{ji})a_{ij}) + (a_{ij}b_{ji})\delta(c_{ij}) \\
&\ \ \ + \delta(c_{ij})(a_{ij}b_{ji})-  \delta(a_{ij}b_{ji})c_{ij} - (a_{ij}b_{ji})\delta(c_{ij}) - \delta(c_{ij})(b_{ji})a_{ij}) - c_{ij}\delta(b_{ji}a_{ij}) \\
&= [(\delta(a_{ij})b_{ji}) + a_{ij}\delta(b_{ji}) - \delta(a_{ij}b_{ji})) + (\delta(b_{ji}a_{ij}) - \delta(b_{ji})a_{ij} - b_{ji}\delta(a_{ij})) , c_{ij} ].
\end{aligned}
$$
Since $\R_{ij} \cap \mathcal{Z} = \left\{0\right\}$, we know that $[(\delta(a_{ij})b_{ji}) + a_{ij}\delta(b_{ji}) - \delta(a_{ij}b_{ji})) + (\delta(b_{ji}a_{ij}) - \delta(b_{ji})a_{ij} - b_{ji}\delta(a_{ij})) , c_{ij} ] = 0$ for all $c_{ij} \in \R_{ij}$.
By Proposition \ref{prop2} it follows that
$$
[\delta(a_{ij})b_{ji} + a_{ij}\delta(b_{ji}) - \delta(a_{ij}b_{ji})] + [\delta(b_{ji}a_{ij}) - \delta(b_{ji})a_{ij} - b_{ji}\delta(a_{ij})] = z \in \mathcal{Z}(\R).
$$ 
If $z= 0$, then $\delta(a_{ij}b_{ji}) = \delta(a_{ij})b_{ji} + a_{ij}\delta(b_{ji})$.
If $z \neq 0$, we multiply by $a_{ij}$ and get
$$
a_{ij}\delta(b_{ji}a_{ij}) - a_{ij}\delta(b_{ji})a_{ij} - a_{ij}(b_{ji}\delta(a_{ij})) = a_{ij}z.
$$
By $({\rm II)}$ we have
\begin{eqnarray*}\label{dif}
\delta(a_{ij}b_{ji}a_{ij}) - \delta(a_{ij})(b_{ji}a_{ij})- a_{ij}\delta(b_{ji})a_{ij} - a_{ij}(b_{ji}\delta(a_{ij})) = a_{ij}z.
\end{eqnarray*}
Now we see that $\delta(a_{ij}b_{ji}a_{ij}) = \delta(a_{ij})(b_{ji}a_{ij})+ a_{ij}\delta(b_{ji})a_{ij} + a_{ij}(b_{ji}\delta(a_{ij}))$. Indeed, note that $p_n(a_{ij}, b_{ji},a_{ij},e_j,...,e_j) = 2a_{ij}b_{ji}a_{ij}$. 
Thus
$$
\begin{aligned}
2\delta(a_{ij}b_{ji}a_{ij}) &= \delta(2a_{ij}b_{ji}a_{ij}) \\
&=\D(p_n(a_{ij}, b_{ji},a_{ij},e_j,...,e_j)) \\
&= p_n(\D(a_{ij}), b_{ji},a_{ij},e_j,...,e_j) + p_n(a_{ij}, \D(b_{ji}),a_{ij},e_j,...,e_j) + p_n(a_{ij}, b_{ji},\D(a_{ij}),e_j,...,e_j) \\
&= p_n(\delta(a_{ij}), b_{ji},a_{ij},e_j,...,e_j) + p_n(a_{ij}, \delta(b_{ji}),a_{ij},e_j,...,e_j) + p_n(a_{ij}, b_{ji},\delta(a_{ij}),e_j,...,e_j)\\
&= (\delta(a_{ij})b_{ji})a_{ij} + a_{ij}(b_{ji}\delta(a_{ij})) + 2a_{ij}\delta(b_{ji})a_{ij}+ (a_{ij}b_{ji})\delta(a_{ij}) + \delta(a_{ij})(b_{ji}a_{ij}) \\
&= \delta(a_{ij})(b_{ji}a_{ij}) - (a_{ij}b_{ji})\delta(a_{ij}) + a_{ij}(b_{ji}\delta(a_{ij})) + a_{ij}(b_{ji}\delta(a_{ij})) \\
&\ \ \ + 2a_{ij}\delta(b_{ji})a_{ij} + (a_{ij}b_{ji})\delta(a_{ij}) + \delta(a_{ij})(b_{ji}a_{ij}) \\
&=2(\delta(a_{ij})(b_{ji}a_{ij})+ a_{ij}\delta(b_{ji})a_{ij} + a_{ij}(b_{ji}\delta(a_{ij}))).  
\end{aligned}
$$
Applying the fact that $\R$ is $2$-torsion free yields that $\delta(a_{ij}b_{ji}a_{ij}) = \delta(a_{ij})(b_{ji}a_{ij})+ a_{ij}\delta(b_{ji})a_{ij} + a_{ij}(b_{ji}\delta(a_{ij}))$.
So $a_{ij}z = 0$. But, by $(4)$ there exist $h \in \R$ such that $zh = e_1 + e_2$ hence $a_{ij} = 0$, which is a contradiction. Therefore $\delta(a_{ij}b_{ji}) = \delta(a_{ij})b_{ji} + a_{ij}\delta(b_{ji}).$
   
\end{proof}

\begin{lemma}\label{lema7}
$\delta$ is a derivation.
\end{lemma}
\begin{proof}
For any $a, b \in \R$, we have
$$
\begin{aligned}
\delta(ab) &= \delta((a_{11} + a_{12} + a_{21} + a_{22})(b_{11} + b_{12} + b_{21} + b_{22})) \\
&= \delta(a_{11}b_{11}) + \delta(a_{11}b_{12}) + \delta(a_{12}b_{12}) + \delta(a_{12}b_{21}) + \delta(a_{12}b_{22}) \\
&\ \ \ + \delta(a_{21}b_{11}) + \delta(a_{21}b_{12}) + \delta(a_{21}b_{21}) + \delta(a_{22}b_{21}) + \delta(a_{22}b_{22}) \\
&= \delta(a)b + a \delta(b)
\end{aligned}
$$
by Lemmas \ref{lema5} and \ref{lema6}.
\end{proof}

\begin{lemma}\label{lema8}
$\tau$ sends the commutators into zero.
\end{lemma}
\begin{proof}
For any $a_1, a_2, \cdots, a_n\in \R$, we get
$$
\begin{aligned}
\tau(p_n(a_{1}, a_{2},...,a_n)) &= \D(p_n(a_{1}, a_{2},...,a_n)) - \delta(p_n(a_{1}, a_{2},...,a_n))\\
&= \sum_{i=1}^{n}p_n(a_{1}, a_{2},...,a_{i-1},\D(a_i),a_{i+1},...,a_n) - \delta(p_n(a_{1}, a_{2},...,a_n))\\
&= \sum_{i=1}^{n}p_n(a_{1}, a_{2},...,a_{i-1},\delta(a_i),a_{i+1},...,a_n) - \delta(p_n(a_{1}, a_{2},...,a_n))\\
\\&= 0.
\end{aligned}
$$
\end{proof}

Let us now assume that $\D\colon  \R \longrightarrow \R$ is a Lie-type derivation of the form $\D = \delta + \tau$, 
where $\delta$ is a derivation of $\R$ and $\tau$ is a mapping from $\R$ into its commutative center $\mathcal{Z}(\R)$, such that 
$\tau(p_n(a_{1}, a_{2},...,a_n)) = 0$ for all $a_1,a_2,...,a_n \in \R$.
Then for any $a_{11}\in \R_{11}$, we see that
$$
\begin{aligned}
e_{2}\D(a_{11})e_2 &= e_2\delta(a_{11})e_2 + e_2\tau(a_{11})e_2 \\
&= e_2\delta(e_1 a_{11})e_2 + e_2\tau(a_{11})e_2\\
&= e_2(\delta(e_1)a_{11} +e_1 \delta(a_{11}))e_2 + e_2\tau(a_{11})e_2 \\
&= e_2(\delta(e_1)a_{11})e_2 + e_2(e_1 \delta(a_{11}))e_2 + e_2\tau(a_{11})e_2 \\
&=(e_2\delta(e_1))(a_{11}e_2) + (e_2e_1) (\delta(a_{11})e_2) + e_2\tau(a_{11})e_2 \\
&=e_2\tau(a_{11})e_2 \in \mathcal{Z}(\R)e_2.
\end{aligned}
$$
Now
$$
\begin{aligned}
e_{1}\D(a_{22})e_1 &=e_1\delta(a_{22})e_1 + e_1\tau(a_{22})e_1 \\
&= e_1\delta(e_2 a_{22})e_1 + e_1\tau(a_{22})e_1\\
&= e_1(\delta(e_2)a_{22} +e_2 \delta(a_{22}))e_1 + e_1\tau(a_{22})e_1 \\
&= e_1(\delta(e_2)a_{22})e_1 + e_1(e_2 \delta(a_{22}))e_1 + e_1\tau(a_{22})e_1 \\
&=(e_1\delta(e_2))(a_{22}e_1) + (e_1e_2) (\delta(a_{22})e_1) + e_1\tau(a_{22})e_1 \\
&= e_1\tau(a_{22})e_1 \in \mathcal{Z}(\R)e_1
\end{aligned}
$$
for all $a_{22} \in \R_{22}$. Furthermore, 
$$
\D(a_{ij}) = (\delta + \tau)(a_{ij}) = \delta(p_n(a_{ij}, e_j,...,e_j)) + \tau(p_n(a_{ij}, e_j,...,e_j)) = p_n(\delta(a_{ij}), e_j,...,e_j) \in \R_{ij}.
$$
This shows the items $(a)$, $(b)$, $(c)$ and the proof of the Theorem \ref{mainthm} is complete.
\vspace{2mm}

\begin{corollary}
Let $\R$ be an unital prime alternative ring with nontrivial idempotent satisfying $(4)$ and 
$\D \colon  \R \longrightarrow \R$ be a multiplicative Lie-type derivation. Then
$\D$ is the form of $\delta + \tau$, where $\delta$ is a derivation of $\R$ and $\tau$ is a 
mapping from $\R$ into its commutative center $\mathcal{Z}(\R)$, such that 
$\tau(p_n(a_{1}, a_{2},...,a_n)) = 0$ for all $a_{1}, a_{2},...,a_n \in \R$ if and only if
\begin{enumerate}
\item[(a)] $e_2\D(\R_{11})e_2 \subseteq \mathcal{Z}(\R) e_2,$
\item[(b)] $e_1\D(\R_{22})e_1 \subseteq \mathcal{Z}(\R) e_1,$
\item[(c)] $\D(\R_{ij}) \subseteq \R_{ij}$, $1 \leq i \neq j \leq 2.$
\end{enumerate}
\end{corollary}

Let us end our work with a direct application to simple alternative rings.

\begin{corollary}
Let $\R$ be an unital simple alternative ring with nontrivial idempotent and $\D \colon  \R \longrightarrow \R$ 
be a multiplicative Lie-type derivation. Then
$\D$ is the form $\delta + \tau$, where $\delta$ is a derivation of $\R$ and $\tau$ is a mapping 
from $\R$ into its commutative center $\mathcal{Z}(\R)$, such that $\tau(p_n(a_{1}, a_{2},...,a_n)) = 0$ for all $a_{1}, a_{2},...,a_n \in \R$ if and only if
\begin{enumerate}
\item[(a)] $e_2\D(\R_{11})e_2 \subseteq \mathcal{Z}(\R) e_2,$
\item[(b)] $e_1\D(\R_{22})e_1 \subseteq \mathcal{Z}(\R) e_1,$
\item[(c)] $\D(\R_{ij}) \subseteq \R_{ij}$, $1 \leq i \neq j \leq 2.$
\end{enumerate}
\end{corollary}
\begin{proof}
It is enough to remark that every simple ring is prime and $\mathcal{Z}(\R)$ is a field.
\end{proof}

\vspace{4mm}


\end{document}